\documentclass[12pt]{amsproc}
\usepackage{amsmath,amsthm}
\usepackage{amsfonts,amssymb}
\newtheorem{thm}{Theorem}[section]

\theoremstyle{remark}

\begin{document}

\title[Characterisation of Fourier transform]
{A Characterisation of the Euclidean Fourier transform on the
Schwartz space}
\author{R. Lakshmi Lavanya}

\address{Department of Mathematics\\ Indian Institute
of Science Education and Research\\Tirupati-517 507}
\email{rlakshmilavanya@iisertirupati.ac.in}

\date{\today}
\thanks{The author is thankful to the anonymous referee of the paper\cite{LT} for suggesting the problem. She thanks Prof. E.K. Narayanan
for critically reading the manuscript, and for his invaluable
suggestions on the earlier versions.}

 \keywords{Fourier transform, Schwartz class
functions} \subjclass[2010]{42A38}
%\thanks{ This article is based on the author's talk in the international
%conference 'Geometric and Harmonic Analysis on homogeneous spaces'
%held in Kerkennah, Tunisia. The author wishes to thank Ali
%%Baklouti for the invitation and warm hospitality during the
%%conference}

\begin{abstract}
We obtain a characterisation of the Fourier transform on the space
of Schwartz class functions on $\mathbb{R}^n.$ The result states
that any appropriately additive bijection of the Schwartz space
onto itself, which interchanges convolution and pointwise products
is essentially the Fourier transform.
\end{abstract}

\maketitle

\section{Introduction}
\setcounter{equation}{0}
 The Fourier transform on various locally
compact groups, and its properties with respect to different
operations on function spaces on these groups are well understood.
The interaction of the Fourier transform with the translations on
the groups, and with certain products on the functions defined on
these groups have been used to obtain characterisations of the
Fourier transform. For more details, refer to \cite{AAM1}-\cite{LT}, and the references therein.\\

We denote by $\mathcal{S}(\mathbb{R}^n),$ the Schwartz class of
rapidly decreasing functions on $\mathbb{R}^n,$ defined as
follows:\\

For a function $f: \mathbb{R}^n \rightarrow \mathbb{C},$ let
$$\|f\|_{\alpha,\beta} : = \sup \limits_{x\in \mathbb{R}^n} \ |x^\alpha \partial^\beta f(x)|,$$
where for multi-indices $\alpha=(\alpha_1,\cdots , \alpha_n),$ and
 $\beta = (\beta_1,\cdots , \beta_n)\in ~\mathbb{N}^n,$ we denote
$x^\alpha = \prod\limits_{j=1}^n x_j^{\alpha_j},$ and $
\partial^\beta = \frac{\partial^{\beta_1}}{\partial
x_1^{\beta_1}}\frac{\partial^{\beta_2}}{\partial x_2^{\beta_2}}
\cdots \frac{\partial^{\beta_n}}{\partial x_n^{\beta_n}}.$ \\

The \textit{Schwartz class} of functions, denoted
$\mathcal{S}(\mathbb{R}^n),$ or simply $\mathcal{S},$ is defined
to be
$$\mathcal{S}(\mathbb{R}^n):=\{f:\mathbb{R}^n \rightarrow \mathbb{C} : f\in \mathcal{C}^\infty(\mathbb{R}^n),
\|f\|_{\alpha,\beta} <\infty \textrm{ \  for \ all \ } \alpha,
\beta \in \mathbb{N}^n\}.$$ Then the space
$\mathcal{C}_c^\infty(\mathbb{R}^n),$ also denoted
$\mathcal{C}_c^\infty,$ of compactly supported smooth functions
defined on $\mathbb{R}^n,$ is a
subspace of $\mathcal{S}(\mathbb{R}^n).$\\

\noindent The topology generated by the family of seminorms
$\{\|\cdot\|_{\alpha,\beta}: \alpha,\beta\in \mathbb{N}^n\}$ makes
$\mathcal{S}(\mathbb{R}^n)$ into a Fr\'{e}chet space over the
complex numbers. Also, $\mathcal{S}(\mathbb{R}^n)$ is closed under
the operations of pointwise and convolution product, where the
convolution of functions in $\mathcal{S}(\mathbb{R}^n)$ is defined
as
$$(f\ast g)(x) = \int\limits_{\mathbb{R}^n} f(x-y) \ g(y) \ dy, \ x\in \mathbb{R}^n.$$
For a function $f\in \mathcal{S}(\mathbb{R}^n),$ its Fourier
transform $\mathcal{F}f$ is defined as
$$\mathcal{F}f(\xi) = (2\pi)^{-n/2} \int\limits_{\mathbb{R}^n} f(x) \ e^{-ix\cdot \xi} \ dx, \ \xi\in \mathbb{R}^n.$$

The space of all continuous linear functionals on
$\mathcal{S}(\mathbb{R}^n)$ is called the \textit{space of
tempered distributions}, and is denoted by
$\mathcal{S}'(\mathbb{R}^n).$ We denote the action of $\varphi\in
\mathcal{S}'(\mathbb{R}^n)$ on a
function $f\in \mathcal{S}(\mathbb{R}^n)$ as $\langle \varphi,f \rangle .$\\

The operations of pointwise multiplication and convolution of
functions in $\mathcal{S}(\mathbb{R}^n)$ can be appropriately
extended to $\mathcal{S}'(\mathbb{R}^n)$ as follows:

\noindent For $f,g \in \mathcal{S}(\mathbb{R}^n)$ and $\varphi \in
\mathcal{S}'(\mathbb{R}^n)$,
\begin{eqnarray*}
  \langle f\cdot \varphi,g\rangle &=& \langle \varphi,f\cdot g\rangle  \\
  \langle f\ast \varphi,g\rangle &=& \langle \varphi, \tilde{f} \ast g\rangle ,
   \end{eqnarray*}
   where $\tilde{f}(x)=f(-x)$ for $x \in \mathbb{R}^n.$
   Then for $f,g\in \mathcal{S}(\mathbb{R}^n)$ and $\varphi \in
   \mathcal{S}'(\mathbb{R}^n),$ we have $f\cdot \varphi \in \mathcal{S}'(\mathbb{R}^n)$ and $f\ast
   \varphi \in \mathcal{S}'(\mathbb{R}^n).$\\

   \noindent The Fourier transform, initially defined on
   $\mathcal{S}(\mathbb{R}^n),$ can be extended to the space
   $\mathcal{S}'(\mathbb{R}^n)$
   via $$\langle \mathcal{F}\varphi,f \rangle = \langle \varphi, \mathcal{F}f \rangle , \ \textrm{for} \ f\in \mathcal{S}(\mathbb{R}^n), \ \varphi \in \mathcal{S}'(\mathbb{R}^n).$$
   The Fourier transform is a topological isomorphism of
   $\mathcal{S}'(\mathbb{R}^n)$ onto itself and satisfies
   \begin{eqnarray*}
      \mathcal{F} (f\cdot \varphi)&=& \mathcal{F}(f)\ast \mathcal{F}(\varphi) \\
      \mathcal{F}(f\ast \varphi)&=& \mathcal{F}(f) \cdot
      \mathcal{F}(\varphi).
   \end{eqnarray*}

In \cite{AAM2} S. Alesker, S. Artstein-Avidan and V. Milman gave a
very interesting characterisation of the Fourier transform on the
Schwartz class of functions on $\mathbb{R}^n.$ The precise
statement of their result is as follows:
\begin{thm}
Assume that $T:\mathcal{S}(\mathbb{R}^n) \rightarrow
\mathcal{S}(\mathbb{R}^n)$ is a bijection which admits a bijective
extension $T': \mathcal{S}'(\mathbb{R}^n) \rightarrow
\mathcal{S}'(\mathbb{R}^n)$ such that for all $f\in
\mathcal{S}(\mathbb{R}^n)$ and $\varphi \in
\mathcal{S}'(\mathbb{R}^n),$ we have
$$T(f\ast \varphi) = T(f) T(\varphi) \textrm{ \ and \ } T(f\cdot \varphi) = T(f) \ast T(\varphi).$$
Then $T$ is essentially the Fourier transform: that is, for some
matrix $B\in GL(n,\mathbb{R})$ with $|det \ B|=1,$ we have either
$Tf= \mathcal{F}(f\circ B)$ or $Tf=\mathcal{F}(\overline{f\circ
B})$ for all functions $f\in \mathcal{S}(\mathbb{R}^n).$
\end{thm}
As the authors of the above result had remarked, the hypotheses of
this result involves only \textit{algebraic} properties of the map
on the class of tempered distributions, whereas the conclusion
states
that the map is essentially the Fourier transform.\\

Motivated by the above result, a characterisation of the Fourier
transform on the Schwartz space of the Heisenberg group was
obtained in \cite{LT}. This result did not involve any hypothesis
in terms of the tempered distributions. The anonymous referee of
\cite{LT} suggested if a characterisation of the Fourier transform
on $\mathbb{R}^n,$ without any assumptions on the tempered
distributions, could be obtained. This paper is an attempt towards
a positive answer to this question.\\

For a function $f:\mathbb{R}^n\rightarrow \mathbb{C},$ the support
of $f,$ denoted $Supp \ f,$ is defined as
$$Supp \ f := Closure({\{x\in \mathbb{R}^n : f(x) \neq 0\}}).$$
%Here $\overline{\phantom{2cm}}$ denotes the closure of the set.

\section{A Characterisation of Fourier transform on $\mathcal{S}(\mathbb{R}^n)$}

We remark that our results are very much influenced by the those
of Alesker et al.\cite{AAM2} and their interesting proofs.\\

\noindent Our main result is the following:

\begin{thm}\label{LaT} Let $T:\mathcal{S}(\mathbb{R}^n) \rightarrow \mathcal{S}(\mathbb{R}^n)$ be a bijection satisfying the following conditions for all functions
$f,g \in \mathcal{S}(\mathbb{R}^n):$
\begin{description}
\item[(a)]  $T(f+\overline{g}) = T(f)+[T(g)]^*,$ where $[Tg]^*(x)
= \overline{Tg(-x)}, \ x\in ~\mathbb{R}^n.$
 \item[(b)] $T(f\cdot g)=
T(f) \ast T(g )$, \item[(c)] $T(f\ast g)= T(g)\cdot T(g).$
\end{description} Then there exists a matrix $B\in GL(n,\mathbb{R}),$ with $|det \ B|=1$ such that
either $Tf = \mathcal{F}(f\circ B)$ for all $f\in
\mathcal{S}(\mathbb{R}^n),$ or $Tf = \mathcal{F}(\overline{f\circ
B})$ for all $f\in \mathcal{S}(\mathbb{R}^n).$
\end{thm}

\begin{proof}
For $f\in \mathcal{S}(\mathbb{R}^n),$ we have $Tf \in
\mathcal{S}(\mathbb{R}^n).$ Since the Fourier transform
$\mathcal{F}$ is a bijection on $\mathcal{S}(\mathbb{R}^n),$ there
exists unique $g\in \mathcal{S}(\mathbb{R}^n)$ with
$Tf=\mathcal{F}g.$ Define a map $U:
\mathcal{S}(\mathbb{R}^n)\rightarrow \mathcal{S}(\mathbb{R}^n)$ as
$Uf:=g$ if $Tf=\mathcal{F}g.$ Then $Tf = \mathcal{F}(Uf)$ for all
$f\in \mathcal{S}.$ The map $U$ is a bijection of $\mathcal{S}$
onto itself and satisfies the following conditions for all
functions $f,g\in\mathcal{S}(\mathbb{R}^n):$
\begin{enumerate}
\item  $U(f+\overline{g}) = U(f)+\overline{U(g)},$ \item $U(f\cdot
g)= U(f) \cdot U(g )$, \item  $U(f\ast g)= U(f)\ast U(g).$
\end{enumerate}
The theorem is a then a consequence of the following result, which gives a precise description of the map $U.$\\

\end{proof}

\begin{thm}\label{LaU} Let $U:\mathcal{S}(\mathbb{R}^n) \rightarrow \mathcal{S}(\mathbb{R}^n)$ be a bijection satisfying the following conditions for all functions
$f,g \in \mathcal{S}(\mathbb{R}^n):$
\begin{enumerate} \item $U(f+\overline{g}) = U(f)+\overline{U(g)},$ \item $U(f\cdot g)=
U(f) \cdot U(g )$, \item $U(f\ast g)= U(f)\ast U(g).$
\end{enumerate} Then there exists a matrix $B\in GL(n,\mathbb{R})$ with $|det \ B| =1$ such that
either $Uf(x) = f(Bx)$ for all $f\in \mathcal{S}(\mathbb{R}^n),$
or $Uf(x) = \overline{f(Bx)}$ for all $f\in
\mathcal{S}(\mathbb{R}^n).$
\end{thm}

\begin{proof} We prove the result in 12 steps.

\noindent For $x_0\in \mathbb{R}^n,$ define $$C(x_0): = \{f \in
\mathcal{S}(\mathbb{R}^n): x_0\in Supp\ f\}.$$\\
\vspace{-0.8cm}

\noindent \textbf{Step 1.} Let $f,g \in \mathcal{S}.$ If $g=1$ on
$Supp \ f,$ then $Ug=1$ on
$Supp \ Uf.$\\

\noindent \textit{Proof of Step 1.} Since $g=1$ on $Supp \ f,$ we
have $f\cdot g =f.$ This gives $Uf= U(f\cdot g) = Uf \cdot Ug,$
and so $Ug=1$ on the set $\{x : Uf(x) \neq 0\}.$\\

\noindent Let $x\in Supp \ Uf$ with $Uf(x)=0.$ Then there is a
sequence $\{x_k\}_{k\in \mathbb{N}} \subseteq ~\mathbb{R}^n$ with
$Uf(x_k) \neq 0$ for all $k$ and $x_k \rightarrow x$ as $k
\rightarrow \infty.$ Since $Uf(x_k) \neq 0,$ we have $Ug(x_k) =1$
for all $k.$ Hence $Ug(x) = \lim \limits_{k \rightarrow \infty} \
Ug(x_k) = 1.$
Thus $Ug=1$ on $Supp \ Uf.$\\

\noindent \textbf{Step 2.}  If $f \in \mathcal{C}_c^\infty,$ then $Uf \in \mathcal{C}_c^\infty.$\\

\noindent \textit{Proof of Step 2.} Choose $f\in
\mathcal{C}_c^\infty$ such that $f(x_0) \neq 0.$ Choose $g\in
\mathcal{S}$ such that $g=1$ on $Supp \ f.$ By Step 1, $Ug=1$ on
$Supp \ Uf.$
Since $Ug\in \mathcal{S},$ we get $Supp \ f$ is compact.\\

%\noindent \textbf{Step 2.} If $g\in \mathcal{S}$ is such that
%$g\equiv 1$ on $Supp \ f,$ for some $f \in \mathcal{S},$ then $Ug
%\equiv
%1$ on $Supp \ Uf.$\\
%
%\noindent Proof of Step 2. $g=1$ on $Supp \ f \Rightarrow f\cdot g
%\equiv f \Rightarrow U(f\cdot g) = Uf \Rightarrow Ug \equiv 1$ on
%$Supp \ Uf.$\\

%%\noindent Step 3. Let $x_0 \in \mathbb{R}^n.$ There exists $g \in \mathcal{S}(\mathbb{R}^n)$
%%with $Ug \in
%%\mathcal{C}_c^\infty$ and $g(x_0) \neq 0.$\\
%%
%%\noindent Proof of Step 3.**************

\noindent \textbf{Step 3.} For any $x_0 \in \mathbb{R}^n,$ there
exists $y_0 \in \mathbb{R}^n$ such
that $Uf\in C(y_0)$ whenever $f \in C(x_0).$\\

\noindent \textit{Proof of Step 3. }Let $E:=\{f\in \mathcal{S}:
f(x_0) \neq 0\}.$ Fix a function $g\in \mathcal{C}_c^\infty$ with
$g(x_0) \neq 0.$ By Step 2, we have $K: = Supp \ Ug$ is compact.\\

\noindent  For $f\in E,$ define $K_f:= K \cap Supp \ Uf.$ For
functions $f_0:=g,f_1,\cdots ,f_k\in E,$ we have $\prod_{j=0}^k \
f_j\not \equiv 0,$ and so $\prod_{j=0}^k \ Uf_j= U(\prod_{j=0}^k \
f_j) \not \equiv 0,$ which gives $\cap _{j=0} ^k \ K_{f_j} \neq
\emptyset.$ This means, the collection $\{K_f: f\in E\}$ of closed
subsets of $K$ has finite intersection property. Since $K$ is
compact, this gives $\bigcap \limits_{f\in E} \ K_f \neq
\emptyset.$ Let $y_0 \in \bigcap
\limits_{f\in E} \ K_f.$\\

\noindent \textbf{Claim.} If $f\in C(x_0),$ then $Uf \in C(y_0).$\\

\noindent \textit{Proof of Claim.} We prove the claim in two
separate
cases.\\

\noindent \textbf{Case 1.} $f(x_0) \neq 0.$\\
 Then
$f$ never vanishes on a neighborhood, say, $V$ of $x_0.$ Let $g\in
\mathcal{S}$ be such that $f\cdot g = 1$ on $V.$ Choose $h\in
\mathcal{S}$ such that $h=1$ on a neighborhood $W$ of $x_0,$ and
satisfies $W\subseteq Supp \ h \subseteq V.$ Since $f\cdot g=1$ on
$Supp \ h,$ by Step 1, we get $U(f\cdot g)=Uf\cdot Ug =1$ on
$Supp\
Uh.$ Since $h\in E,$ by definition, $y_0 \in Supp \ Uh .$ This implies $Uf(y_0) \neq 0,$ and hence $Uf\in C(y_0).$\\

\noindent Note that all our arguments till now can be applied to
the map $U^{-1}$ as well, and so we have proved that $f(x_0)\neq
0$ if and
only if $Uf(y_0) \neq 0.$\\

\noindent A function $f\in \mathcal{S}$ is said to satisfy the
condition $(\star)$ if the following holds:
$$(\star) \ \ \ \ \ \ \ \ \ \ f(x_0)=0 \textrm{ \ if \ and\  only \ if \ } Uf(x_0)=0.\hspace{4cm}$$
By the above discussion, we have that all functions in
$\mathcal{S}$ satisfy condition $(\star).$\\

\noindent \textbf{Case 2.} $f(x_0) =0.$\\

\noindent Suppose $Uf \not \in C(y_0).$ Then there is a
neighbourhood, say $W,$ of $y_0$ such that $Uf$ vanishes
identically on $W.$ Let $h\in \mathcal{S}$ with $Supp \ h
\subseteq W,$ and $h(y_0) \neq 0.$ There exists unique function $g
\in \mathcal{S}$ with $Ug=h.$ Then $U(f\cdot g) = Uf\cdot Ug = Uf
\cdot h \equiv 0.$ This gives $f\cdot g
\equiv 0.$ \\

\noindent On the other hand, since $Ug(y_0) = h(y_0) \neq 0,$ by
Condition $(\star),$ we have $g(x_0) \neq 0,$ and so g is never
zero near $x_0.$ Since $x_0\in Supp \ f,$
 this implies $f\cdot g \not \equiv 0,$ a contradiction. Thus $Uf \in
C(y_0).$\\

%\noindent \textit{Proof of Step 3. }Let $x_0 \in \mathbb{R}^n,$
%and choose $g\in
%\mathcal{C}_c^\infty,$ with $g(x_0) \neq 0.$ Let $K = Supp \ Ug.$\\
%
%
%Let $f_0:=g.$ For any  $f \in \mathcal{S},$ let $K_f: = K \cap
%Supp \ Uf_0.$ Then $K_f$ is a closed subset of $K.$ Let
%$f_1,f_2,\cdots ,f_k \in C(x_0).$ Since $x_0 \in \cap_{j=0}^{k}
%Supp \ f_j,$ we have $\prod\limits_{j=0}^{k} \ f_j \not\equiv 0.$
%So $\prod \limits_{j=0}^{k} \ U(f_j) = U(\prod \limits_{j=0}^{k} \
%f_j) \not\equiv 0,$ thus ensuring that $\bigcap \limits_{j=0}^{k}
%K_{f_j}$ is nonempty. Hence the collection $\{K_f: f\in C(x_0)\}$
%of closed subsets of $K$ has finite intersection property. Since
%$K$ is compact, this guarantees $\cap_{f\in C(x_0)} \ K_f \neq
%\emptyset.$ This means $\cap_{f\in C(x_0)} \ Supp \ Uf \neq
%\emptyset,$ proving the existence of atleast one point $y_0\in
%\mathbb{R}^n$ with the required
%property.\\

\noindent \textbf{Step 4.} Define a map $A: \mathbb{R}^n
\rightarrow \mathbb{R}^n$ as follows: $Ax=y$ if $Uf\in C(y)$
whenever $f\in C(x).$ Then the map $A$ is well-defined.\\

\noindent \textit{Proof of Step 4.} Suppose for some $x_0 \in
\mathbb{R}^n,$ we have $Ax_0 = y_1$ and $Ax_0 = y_2$ with $y_1\neq
y_2.$ Let $V_1$ and $V_2$ be disjoint neighborhoods of $y_1$ and
$y_2,$ respectively. There exists functions $g_1$ and $g_2$ in
$\mathcal{S}$ which are supported
 in $V_1$ and $V_2,$ respectively, such that $g_1(y_1) \neq 0$ and $g_2(y_2) \neq 0.$ Let $f_1,f_2 \in \mathcal{S}$ with $Uf_1=g_1$ and
 $Uf_2=g_2.$ Then $0\equiv g_1\cdot g_2 =Uf_1 \cdot Uf_2 = U(f_1 \cdot f_2)$ and so $f_1\cdot f_2 \equiv
 0.$\\

 On the other hand, as $g(y_j) = Uf_j(y_j)\neq 0$ for $j=1,2,$ we
 have by
Condition $(\star)$ that $f_j(x_0) \neq
 0$
 for $j=1,2,$ which is in contradiction to $f_1\cdot f_2\equiv 0.$\\

\noindent \textbf{Step 5.} $A:\mathbb{R}^n \rightarrow \mathbb{R}^n$ is a bijection.\\

\noindent \textit{Proof of Step 5.} The hypotheses of the theorem
hold good for the map $U^{-1}$ as well. Applying the preceding
steps to the map $U^{-1}$ gives rise to a well-defined function,
say,
$B:\mathbb{R}^n \rightarrow \mathbb{R}^n.$ Then $B = A^{-1}$, proving that $A$ is bijective.\\

\noindent Our observations can be summarised as:
$$A(Supp \ f) = Supp \ Uf,\textrm{ \ for \ all \ } f\in \mathcal{S}.$$

%\noindent \textbf{Step 6.} For any $f\in \mathcal{S}$ and $x_0 \in
%\mathbb{R}^n,$ $f(x_0)=0$ if and only if
%$Uf(Ax_0)=0.$\\

%\noindent \textit{Proof of Step 6.} Suppose $f(x_0) \neq 0.$ Then
%$f$ never vanishes on a neighborhood, say, $V$ of $x_0.$ Let $g\in
%\mathcal{S}$ be such that $f\cdot g = 1$ on $V.$ Choose $h\in
%\mathcal{S}$ such that $h=1$ on a neighborhood $W$ of $x_0,$ and
%satisfies $W\subseteq Supp \ h \subseteq V.$ Since $f\cdot g=1$ on
%$Supp \ h,$ by Step 1, $U(f\cdot g)=Uf\cdot Ug =1$ on $Supp\
%Uh,$ which contains $A(W),$ and hence contains $Ax_0.$ Thus $Uf(Ax_0) \neq 0.$\\
%
%
%A similar argument for $U^{-1}$ completes a proof of this step.\\

%Then $f$ never vanishes near $x,$ which gives the existence of a
%function $g\in \mathcal{S}$ satisfying $g=\frac{1}{f}$ near $x.$
%This gives $f\cdot g=1$ near $x.$ Using Step 2, we get $Uf\cdot Ug
%=1$ near $Ax.$ In particular, $Uf$ is
%never-vanishing near $Ax.$ Applying the argument to the map $U^{-1}$ completes this a proof of this step.\\

\noindent \textbf{Step 6.} The map $A$ is a homeomorphism of $\mathbb{R}^n$ onto itself.\\

\noindent \textit{Proof of Step 6.} Suppose not. Then there exist
$x\in \mathbb{R}^n,$ and sequence $\{x_k\}$ in $\mathbb{R}^n$ with
$x_k
\rightarrow x$ as $k \rightarrow \infty,$ but $Ax_k$ does not converge to $Ax.$\\

\noindent Let $V$ be a neighborhood of $Ax$ such that $Ax_k
\not\in V$ for any $k.$ Let $h\in \mathcal{S}$ with $Supp \
h\subseteq V,$ and $h(Ax) = 1.$ Let $g\in \mathcal{S}$ be such
that $Ug=h.$ Then $Ax_k \not \in Supp \ Ug$ for any $k,$ and so
$x_k \not\in Supp \ g$ for any $k.$ This gives
 $g(x_k) = 0$ for all $k,$ implying $g(x)= 0,$ which is not
possible by Condition $(\star),$ since $Ug(Ax) =1.$\\

We observe that the above argument holds good when the maps $U$
and $A$ are replaced with $U^{-1}$ and $A^{-1}$ respectively,
yielding that
$A: \mathbb{R}^n \rightarrow \mathbb{R}^n$ is a homeomorphism.\\

%\noindent Proof of Step 8. Let $y \in \mathbb{R}^n.$ Choose $g\in \mathcal{S}$ such that $y \in
%Supp \ g.$ There exists unique $f\in \mathcal{S}$ with $Uf=g.$ Since $y\in
%Supp \ g,$ $Uf \not\equiv 0.$ So $f\not \equiv 0,$ ensuring the
%existence of $x\in \mathbb{R}^n$ with $x\in Supp \ f.$ This gives $f\in
%C(x).$*********************\\
%
%\noindent Step 9. $A$ is one-to-one.\\
%
%\noindent Proof of Step 9. *********************\\

\noindent \textbf{Step 7.} The map A satisfies $A(x+y) = Ax+Ay$ for all $x,y \in \mathbb{R}^n.$\\

\noindent \textit{Proof of Step 7.} Suppose $A(x+y) \neq Ax+Ay$ for some $x,y\in \mathbb{R}^n.$\\

\noindent Then there exist disjoint neighborhoods $V_{x+y},V_{xy}$
with $A(x+y) \in ~V_{x+y}$ and $Ax+Ay \in V_{xy}.$ By continuity
of the map $A,$ this gives rise to a neighborhood $W_{x+y,1}$ of
$(x+y)$ with $A(W_{x+y,1}) \subseteq V_{x+y}.$ By continuity of
addition in $\mathbb{R}^n,$ we get neighborhoods $W_{x,1},
W_{y,1}$ of $x$ and $y,$ respectively, such that $W_{x,1}+W_{y,1}
\subseteq W_{x+y,1}.$ Thus
\begin{eqnarray} \vspace{-1cm} A(W_{x,1}+W_{y,1}) &\subseteq&
A(W_{x+y,1}) \subseteq V_{x+y}.\label{eq1}
\end{eqnarray}

On the other hand, by continuity of addition in $\mathbb{R}^n,$
$Ax+Ay \in V_{xy}$ gives neighborhoods $V_{x,2},V_{y,2}$ such that
\begin{eqnarray} Ax\in V_{x,2}, \ Ay\in V_{y,2} \textrm{ \ and \ }
V_{x,2}+V_{y,2}\subseteq V_{xy}.\label{eq2} \end{eqnarray}
 This
implies there exist neighborhoods $W_{x,2}, W_{y,2}$ of $x$ and
$y,$ respectively, with $A(W_{x,2})\subseteq V_{x,2}$ and
$A(W_{y,2}) \subseteq V_{y,2}.$\\

Define $W_x= W_{x,1}\cap W_{x,2}, \ W_y = W_{y,1}\cap W_{y,2}.$
Then
\begin{eqnarray}
  A(W_x)& \subseteq & A(W_{x,2}) \subseteq V_{x,2} =V_x \ (say) \nonumber \\
   A(W_y)& \subseteq & A(W_{y,2}) \subseteq V_{y,2} =V_y \ (say)\nonumber \\
   A(W_x+W_y)& \subseteq & A(W_{x,1}+W_{y,1})  \subseteq A(W_{x+y,1}) \subseteq V_{x+y} \label{eq3}
\end{eqnarray}
Choose $f_x,f_y \in \mathcal{S}$ such that $Supp \ f_x \subseteq
W_x, Supp\ f_y \subseteq W_y$ and $f_x \ast ~f_y \not \equiv~0.$
Let $g_x = Uf_x$ and $g_y = Uf_y.$ Then $U(f_x\ast f_y) = g_x \ast
g_y \not\equiv 0.$\\

\noindent We have
\begin{eqnarray}
 \nonumber Supp(g_x \ast g_y) &=&Supp \ U(f_x \ast f_y)  \subseteq  Supp \ Uf_x +Supp \ Uf_y \\
\nonumber &= & A(Supp \ f_x)+A(Supp \ f_y)  \subseteq AW_x+AW_y
\\ & \subseteq & V_{x,2}+V_{y,2} \subseteq V_{xy} \textrm{   \ \ (by \
\ref{eq2})} \label{eq4}
    \end{eqnarray}
But $Supp \ (f_x  \ast f_y) \subseteq Supp \ f_x + Supp \ f_y
\subseteq W_x+W_y.$ By $(\ref{eq3}),$ this gives
 \begin{eqnarray}
 \nonumber Supp(g_x \ast g_y) &= &Supp (Uf_x \ast Uf_y) = Supp \ U(f_x \ast f_y)\\
  & =& A(Supp (f_x  \ast f_y)) \subseteq  A(W_x+W_y)\subseteq
 V_{x+y}. \label{eq5}
 \end{eqnarray}
From $(\ref{eq4})$ and $(\ref{eq5})$, we get $$Supp(g_x \ast g_y)
 \subseteq V_{xy} \cap V_{x+y} = \emptyset.$$ This gives $g_x\ast
g_y \equiv 0,$ a contradiction.
This proves the additivity of the map $A.$\\

\noindent \textbf{Step 8.} The map $A:\mathbb{R}^n \rightarrow
\mathbb{R}^n$ is a continuous additive bijection, and so also real
linear. Hence it is given by an invertible matrix,
which also we denote by $A.$\\

%\noindent \textbf{Step 10.} If $f\in \mathcal{S}$ is such that $f(0)=0,$ then $Uf(0) = 0.$ \\
%
%\noindent Proof of Step 10. Let $f\in \mathcal{S}$ with $f(0)= 0.$\\
%Case(a). $0 \not\in Supp(f).$ Then $0(=A(0)) \not \in Supp \ Uf,$
%yielding $Uf(0)=~0.$\\
%
%\noindent Case(b). $0\in Supp \ f.$ Then there exists a sequence
%$x_n \rightarrow 0$ as $n\rightarrow \infty$ such that $f(x_n)
%\neq 0$ for any $n.$ Then $x_n \in Supp \ f,$ and so $Ax_n \in
%Supp \ Uf$ for all $n.$ Then $A(0)=0 \in Supp \ Uf.$ If $Uf(0)\neq
%0,$ then $Uf$ is nonvanishing on a neighborhood of $0,$ implying
%$f$ is so on a neighborhood of $0.$ But as $f(0)=0,$ this is not
%possible. Thus
%$Uf(0)=0.$\\

\noindent \textbf{Step 9.} 'Extension' of the map $U$ to scalars.\\

\noindent \textit{Illustration of Step 9.} For $f,g \in
\mathcal{S},$ and $c(\neq 0)\in \mathbb{C},$ we have
$$ U(cf)(x)\ Ug(x) = U(cfg)(x) = U(f)(x) \ U(cg)(x) , \ x\in \mathbb{R}^n.$$
Let $h \in \mathcal{S}$ be such that $Uh(x) \neq 0$ for any $x\in
\mathbb{R}^n.$ Then we have
\begin{eqnarray*}
 U(cf)(x)  &=& \frac{U(ch)(x)}{Uh(x)} \ {Uf}(x) \textrm{ \ for \ all \ } f\in \mathcal{S}\\
   &=& m(c,x) \ Uf(x) \textrm{ \ (say)}.
\end{eqnarray*}
Thus $U(cf)(x) = m(c,x)\ Uf(x),$ for all $x\in \mathbb{R}^n.$ By
definition, the function $m(\cdot,\cdot)$ is continuous in the
second variable as a function
of $x\in \mathbb{R}^n.$\\

\noindent \textbf{Claim.} The function $m(\cdot,\cdot)$ is
independent of the second
variable.\\
\noindent \textit{Proof of Claim.} For $f,g \in \mathcal{S}, \ c
\in \mathbb{C},$ and $x\in \mathbb{R}^n,$ we have
\begin{eqnarray*}
 U(cf\ast g) (x) &=& U(f\ast cg) (x)  \\
 (U(cf)\ast Ug)(x) &=& (Uf \ast U(cg))(x) \\
   \int\limits_{\mathbb{R}^n} m(c,x-y)\ Uf(x-y) \ Ug(y) \ dy&=&  \int\limits_{\mathbb{R}^n} Uf(x-y) \ m(c,y) \ Ug(y) \ dy \\
\end{eqnarray*}
As the above equation holds good for all functions $f,g \in
\mathcal{S},$ we have for all $F,G \in \mathcal{S},$
   $$\int\limits_{\mathbb{R}^n} [m(c,x-y)-m(c,y)] \ F(x-y) \ G(y) \ dy =0, \ x\in \mathbb{R}^n.$$
   Fix $x\in \mathbb{R}^n.$ Let $G\in \mathcal{S}$ with $G =1$ on $B(0,r),$ where $B(0,r)$ is
   the open ball in $\mathbb{R}^n,$ centered at the origin and with radius
   $r.$
   Then for all functions $F \in \mathcal{C}_c^\infty$ with $Supp \ F
   \subseteq B(x,r),$ we have
\begin{eqnarray*}
  \int\limits_{B(0,r)} [m(c,x-y)-m(c,y)] \ F(x-y) \ dy &=&0. \\
  \textrm{Thus \hspace{2.5cm} } m(c,x-y) -
  m(c,y) &=& 0 \textrm{\ for all \ } y \in B(0,r).
  \end{eqnarray*}
by the continuity of the map $m(\cdot,\cdot)$ in the second
variable. This gives in particular, $m(c,x) =m(c,0).$ As $x$ was
arbitrary, the above argument gives that the function $m(c,x)$ is
independent of the second variable $x\in \mathbb{R}^n.$ We define
$$m(c): = m(c,0).$$

\noindent \textbf{Step 10.}  The map $m :\mathbb{C} \rightarrow
\mathbb{C}$ is an additive and multiplicative bijection, which
maps $\mathbb{R}$ onto $\mathbb{R},$ and hence we have either
$m(a) = a$ for all $a\in \mathbb{C},$ or $m(a) = \overline{a}$ for
all
$a\in \mathbb{C}.$\\

\noindent \textit{Proof of Step 10.} Let $g\in \mathcal{S}, \
a,b\in \mathbb{C}$ with $g(x) \neq 0$ for any $x\in \mathbb{R}^n.$
Then $Ug(y) \neq 0$
for any $y \in \mathbb{R}^n.$\\

Suppose $m(a) =m(b)$ for some $a,b\in \mathbb{C}.$ Then $$U(ag)
(x) = m(a) \ Ug(x)= m(b) \ Ug(x) =
U(bg)(x) , \ x\in \mathbb{R}^n.$$ Since $U$ is a bijection, this gives $a=b.$\\

\noindent By hypothesis(1), we have
\begin{eqnarray*}
 m(a+\overline{b})\ Ug(x)  &=& U((a+\overline{b})g)(x)
=U(ag+\overline{b}g)(x) \\
   &=& U(ag)(x) + \overline{U(b\overline{g})}(x) = (m(a) +m(\overline{b} )) \ Ug(x).
\end{eqnarray*}
Since Ug is never zero, we get $m(a+\overline{b}) =
m(a)+\overline{m(b)}.$ In particular, $m(\overline{a})
=\overline{m(a)}$ for all $a\in \mathbb{C}.$\\

\noindent Now, hypothesis(2) gives
$$m(ab) Ug(x) = U(ab g)(x) = m(a)U(b g)(x) = m(a) m(b) Ug(x).$$
Again, since $Ug$ is nowhere vanishing, we get $m(ab)=m(a)m(b)$
for all $a,b\in \mathbb{C}.$\\

%\noindent \textbf{Step 13.} For $f\in \mathcal{S},$ we denote by
%$R(f),$ the range of the function
%$f.$ Then $R(Uf)\cup \overline{R(Uf)} = R(f) \cup \overline{R(f)}.$\\
%
%\noindent Proof of Step 13.

%For any $c\in C,$ with $c\not \in
%R(f),$ we have, that the function $(f-c)$ is never vanishing on
%$\mathbb{R}^n,$ and hence so is $U(f-c).$ Thus $Uc \not \in
%R(Uf).$ This
%means $c$ or $\overline{c} \not\in R(Uf).$**************\\

%\noindent \textbf{Step 13.} If $f(x_0)=0$ for some $x_0 \in \mathbb{R}^n,$ then $Uf(x_0) = 0.$\\
%
%\noindent Proof of Step 13. Suppose not. Then $Uf$ is never
%vanishing on a neighborhood of $Ax_0,$ and hence so is $f$ near
%$x_0,$ which is a
%contradiction to $f(x_0)=0.$\\

\noindent \textbf{Step 11.} For $f\in \mathcal{S},$ and $x_0 \in
\mathbb{R}^n,$ we have
$Uf(Ax_0) = m(f(x_0)).$\\

\noindent \textit{Proof of Step 11.} As before, choose $g \in
\mathcal{S}$ such that $g(x) \neq 0$
for any $x \in \mathbb{R}^n.$ Then $Ug(y) \neq 0$ for any $y \in \mathbb{R}^n.$ \\

\noindent Define $$h(x): = f(x_0) \ g(x) - f(x) \ g(x), \ x\in
\mathbb{R}^n.$$ Then $h\in \mathcal{S}$ and $h(x_0) = 0.$ By
Condition $(\star),$ we have $Uh(Ax_0) =0.$ This gives
\begin{eqnarray*}
0=Uh(Ax_0) &=& U(f(x_0) \cdot g - f \cdot g)(Ax_0)\\
 &=& m(f(x_0)) \
Ug(Ax_0) - Uf(Ax_0) \ Ug(Ax_0). \end{eqnarray*}
 Since $Ug$ is never zero, this gives
$Uf(Ax_0) = m(f(x_0)).$\\

\noindent Since $B=A^{-1},$ using Step 10, we get that either
$Uf(x_0) = f(Bx_0)$ or $Uf(x_0) = \overline{f(Bx_0)}.$\\

\noindent Thus we get that the map $U$ is as claimed by our
theorem. It remains to show that $|det \ B|=1.$\\

 %Let $B=A^{-1}.$ Then either $Uf(x) =
%f(Bx)$ for all $x\in ~\mathbb{R}^n,\\ f\in ~\mathcal{S},$ or
%$Uf(x) =\overline{f(Bx)}$ for all $x\in \mathbb{R}^n, f\in \mathcal{S}.$\\
%
%\noindent Proof of Step 13. Let $f(x_0) = c.$ Applying the Step 14
%to the function $(f-c),$ we get $Uf(x_0) = c$ or $Uf(x_0) =
%\overline{c}.$ The hypothesis of our result leads to $Uf$ being of
%the claimed form.\\

\noindent \textbf{Step 12.} The matrix $B$ satisfies $|det \ B| =1.$\\

\noindent \textit{Proof of Step 12.} We have
\begin{eqnarray*}
 (f\ast g) (Bx) = U(f\ast g)(x)  &=& (Uf \ast Ug)(x) \\
   &=& \int\limits_{\mathbb{R}^n} Uf(x-y) \ Ug(y) \ dy \\
   &=& \int\limits_{\mathbb{R}^n} f(B(x-y)) \ g(By) \ dy\\
   &=& |det \ B|^{-n} \ \int\limits_{\mathbb{R}^n} f(Bx-y) \ g(y) \ dy
\end{eqnarray*}
Thus $|det \ B|=1,$ proving our result.
\end{proof}

%
%\begin{center}
%{\bf Acknowledgments}
%\end{center} The author is thankful the anonymous referee
%of \cite{LT} for asking if a characterisation of the Fourier
%transform is possible using only the Schwartz class functions, and
%to Prof. E.K. Narayanan for useful comments.

\end{document}